\documentclass[11pt]{llncs}
\usepackage{amsmath,amssymb}
\addtocounter{MaxMatrixCols}{1} 
\usepackage{graphicx,booktabs,balance,enumitem}
\usepackage{pstricks,pst-node,pst-tree}

\renewcommand{\tilde}[1]{\widetilde{#1}}
\newcommand{\bigO}[1]{{\mathcal{O}}\left( #1 \right)}
\begin{document}
\pagestyle{plain}
\title{A new method for constructing small-bias spaces from Hermitian
  codes}

\author{Olav Geil\inst{1}%
\and Stefano Martin\inst{1}%
    \and Ryutaroh Matsumoto\inst{1,2}}%
\authorrunning{Geil, Martin, Matsumoto}
\tocauthor{Olav Geil, Stefano Martin, Ryutaroh Matsumoto}
\institute{Department of Mathematical Sciences, 
    Aalborg University, Denmark\\
    \email{olav@math.aau.dk},\\
    \email{stefano@math.aau.dk},\\
        \and
    Department of Communications and Integrated Systems,
    Tokyo Institute of Technology, Japan\\
    \email{ryutaroh@rmatsumoto.org}}

\maketitle  

\begin{abstract}
We propose a new method for constructing small-bias spaces through  a combination of Hermitian codes. For a class of parameters our
multisets are much faster to construct than what can be achieved by use of
the traditional algebraic geometric code construction. So, if speed is
important, our construction is competitive with all other known
constructions in that region. And if speed is not a matter of
interest the small-bias
spaces of the present paper still perform better than the ones related
to norm-trace codes reported in~\cite{MP}.

\end{abstract}

\section{Introduction}\label{secintro}Let $\vec{X}=(X_1, \ldots, X_k)$ be a random vector that takes on
values in ${\mathbb{F}}_2^k$. As shown by Vazirani \cite{vazirani} the
variables $X_1, \ldots , X_k$ are independent and uniformly
distributed if and only if 
\begin{equation}
{\mbox{Prob}}\left(\sum_{i\in T}X_i=0\right) ={\mbox{Prob}}\left(\sum_{i\in
      T}X_i=1\right) =\frac{1}{2} \label{eqzerobal}
\end{equation}
holds for every non-empty set of indexes $T\subseteq \{1, \ldots , k
\}$. In particular, if~(\ref{eqzerobal}) is to hold for a  space ${\mathcal{X}} \subseteq {\mathbb{F}}_2^k$
then necessarily 
${\mathcal{X}}$ must be equal to ${\mathbb{F}}_2^k$. There is a need for much smaller
 spaces ${\mathcal{X}} \subseteq {\mathbb{F}}_2^k$ with 
statistical properties close to that of~(\ref{eqzerobal}). In the
following by a  space we will mean a multiset ${\mathcal{X}}$
with elements from ${\mathbb{F}}_2^k$ (this we write ${\mathcal{X}}
\subseteq {\mathbb{F}}_2^k$).  The multiset ${\mathcal{X}}$ is made into
a probability space by adjoining to each element $\vec{x} \in
{\mathcal{X}}$ the probability $p(\vec{x})=i(\vec{x})/|{\mathcal{X}}|$
where $i(\vec{x})$ denotes the number of times $\vec{x}$ appears in
${\mathcal{X}}$. As a measure for describing how close a given 
space ${\mathcal{X}}$ is to the above situation with respect to
randomization, 
Naor and Naor \cite{naor}, and Alon et.\ al.\ \cite{alon} introduced
the concept of  
$\epsilon$-biasness \cite[Def.\ 3]{naor}. (See also \cite{zucker}).
\begin{definition}\label{defepsilon}
A multiset ${\mathcal{X}} \subseteq {\mathbb{F}}_2^k$ is called an 
$\epsilon$-bias space if
\begin{equation}
\frac{1}{|{\mathcal{X}} |} \left| \sum_{\vec{x} \in {\mathcal{X}}}(-1)^{\sum_{i\in T}x_i} \right| \leq \epsilon \label{eqepsilon}
\end{equation}
holds for every non-empty index set $T\subseteq \{1, \ldots , k\}$. 
\end{definition}
Clearly, the $\epsilon$ in Definition~\ref{defepsilon} can be taken to
be a 
number between $0$ and $1$. Good randomization properties are
achieved when $\epsilon$ is close to $0$ as~(\ref{eqepsilon})
becomes~(\ref{eqzerobal}) when $\epsilon=0$. Multisets with $\epsilon$ small are called
small-bias spaces. They are useful as sample spaces in applications such as automated
theorem proving, derandomization of algorithms, program verification, and testing of
combinatorial circuits. Rather than saying that a multiset is an $\epsilon$-bias space we will
often just say that it is $\epsilon$-biased. Another name for
$\epsilon$-bias space is 
$\epsilon$-bias set \cite[Def.\ 1]{hermitian} and \cite[Def.\
1.1]{MP}. This notion may be a little misleading as the item under
consideration is actually a multiset. \\
One way of constructing small-bias spaces is through the use of
error-correcting codes.
\begin{definition}\label{defepsiloncode}
A binary $[n,k]$ code is said to be $\epsilon$-balanced if every non-zero
code word $\vec{c}$ satisfies $$ \frac{1-\epsilon}{2} \leq \frac{w_H(\vec{c})}{n}
\leq \frac{1+\epsilon}{2}.$$
Here $[n,k]$ means that the code is linear, of dimension $k$ and
length $n$. Further, $w_H$ denotes the Hamming weight.
\end{definition}
There is a simple direct translation \cite{alon}
 between the concepts described in
Definition~\ref{defepsilon} and Definition~\ref{defepsiloncode}: 
\begin{theorem}\label{thedirect}
Let $G$ be a generator matrix for an $\epsilon$-balanced binary
$[n,k]$ code. The columns of $G$ constitute an $\epsilon$-bias space
${\mathcal{X}} \subseteq {\mathbb{F}}_2^k$ of size 
$n$. Similarly, using the elements of an $\epsilon$-bias space ${\mathcal{X}}$ as
columns of a generator matrix an $\epsilon$-balanced code is derived.
\end{theorem}
The following example illustrates the above theorem. It also shows
why it is important in Definition~\ref{defepsilon} to work with multisets rather than sets. 
\begin{example}
Consider the matrix
$$G=\left[ \begin{array}{cccccccccccc}
0&1&0&1&0&1&0&1&0&0&0&0\\
0&0&1&1&0&0&1&1&0&0&0&0\\
1&1&1&1&1&1&1&1&0&0&0&0
\end{array}\right] .
$$
The code having $G$ as a generator matrix is $\epsilon$-balanced with
$\epsilon= 1/3$ and indeed the multiset made from the columns of $G$
is $\epsilon=1/3$ biased. Treating the columns as a set (rather than a
multiset) we derive
$${\mathcal{X}}^\prime=\{(0,0,1),(1,0,1),(0,1,1), (1,1,1), (0,0,0)\}.$$
The smallest value of $\epsilon$ for which ${\mathcal{X}}^\prime$ is
$\epsilon$-biased is $\epsilon=3/5$.
\end{example}
A standard construction from~\cite{alon} tells us how to make
small-balanced codes (meaning $\epsilon$-biased codes with $\epsilon$
small): 
\begin{theorem}\label{theconc}
Let $q=2^s$ for some integer $s$ and consider a $q$-ary $[N,K,D]$ code
$C$. Let $C_s$ be the (binary) $[2^s,s]_2$ Walsh-Hadamard code, $s\geq
1$. The
concatenated code derived by using $C$ as outer code and $C_s$ as
inner code is an $\epsilon=(N-D)/N$-balanced binary code of length $n=N2^s$ and
dimension $k=Ks$.   
\end{theorem}
\begin{proof}
The result relies on the fact that every non-zero codeword of $C_s$
contains exactly as many $0$s as $1$s.
\end{proof}
The literature contains various examples of small-bias spaces that
cannot all be compared to each other. We refer to~\cite[Sec.\ 1]{hermitian} for
more details. In the following we will concentrate on 
important families of multisets for which comparison can be made. 
We remind the reader of how bigO notation works when given functions
of multiple variables. In our situation we have real valued positive functions $f_i(x,y),
i=1,2$ where $x$ can take on any value in ${\mathbb{Z}}^+$ but for
every fixed choice of $x$ the variable $y$ can
only take on values in an interval $I(x)\subseteq {\mathbb{R}}^+$. By
$f_1(x,y)=\bigO{f_2(x,y)}$ we mean that a witness $(C,\kappa)$ exists
such that for all $x$ with  $\kappa < x$ and all $y \in I(x)$ it holds
that $f_1(x,y) \leq C
f_2(x,y)$. We are interested in upper bounding the size of
${\mathcal{X}}$ which will be done in terms of bigO estimates as
above. At the same time we are interested in lower bounding the length
of the words in the multiset ${\mathcal{X}}$. Such estimates are
described using bigOmega notation. We remind the reader that by
definition $f(x)=\Omega (g(x))$ if and only if $g(x)=\bigO{f(x)}$. As we are only interested in
bigOmega estimates the meaning of $k$ changes accordingly. We have the
following results:
\begin{itemize}
\item Using Reed-Solomon codes as outer codes in Theorem~\ref{theconc}
  one achieves \cite{alon,hermitian} for all possible choices of
  $\epsilon$ and $k$
  $$\mathcal{X} \subseteq  {\mathbb{F}}_2^{\Omega(k)}, {\mbox{\ \ \ }}
  | {\mathcal{X}}|=
\bigO{\frac{k^2}{\epsilon^2\log^2(k/\epsilon)}}.$$ This is
  called the RS-bound.
\item Let $P_1, \ldots , P_{{\mathcal{N}}-1},Q$ be rational places of an algebraic
  function field over ${\mathbb{F}}_q$ and denote by $g$ the
  genus. Assume ${\mathcal{N}}=(\sqrt{q}-1)g$. That is, we assume that
  the function
  field 
  attains the Drinfeld-Vladut bound. Using codes
  $C_{\mathcal{L}}(U=P_1+\cdots+P_{{\mathcal{N}}-1},mQ)$ with $g<m$ as outer codes one gets for all
  $\epsilon$ and $k$  (see Section~\ref{sechermitian} for a discussion) 
$$\mathcal{X} \subseteq  {\mathbb{F}}_2^{\Omega(k)}, {\mbox{\ \ \ }}
  | {\mathcal{X}}|=
\bigO{\frac{k}{\epsilon^3\log(1/\epsilon)}}.$$ This result which
is in the folklore is known as the
AG-bound.
\item Using Hermitian codes with $m<g$ as outer codes one
  achieves \cite{hermitian} for $\epsilon \geq k^{-\frac{1}{2}}$ 
\begin{equation}
\mathcal{X} \subseteq  {\mathbb{F}}_2^{\Omega(k)}, {\mbox{\ \ \ }}
  | {\mathcal{X}}|=
\bigO{\left(
      \frac{k}{\epsilon^2\log(1/\epsilon)}\right)^{\frac{5}{4}}}.\label{eqBT}
\end{equation}
 This we
  call the BT-bound after the authors of~\cite{hermitian}, Ben-Aroya and Ta-Shma.
\item Using in larger generality Norm-Trace codes of low dimension as outer codes one
  achieves \cite{MP} for $l=4, 5, \ldots$ and  $\epsilon \geq
  k^{-\frac{1}{\sqrt{l}}}$ (see Section~\ref{normtrace})
$$\mathcal{X} \subseteq  {\mathbb{F}}_2^{\Omega(k)}, {\mbox{\ \ \ }}
  | {\mathcal{X}}|=
\bigO{\left(\frac{k}{\epsilon^{l-\sqrt{l}}\log(1/\epsilon)}\right)^{\frac{l+1}{l}}}.$$
Here, $l=4$ corresponds to the Hermitian case described in~\cite{hermitian}.
\item The Gilbert-Varshamov bound also applies to the small-bias
  spaces (as usual in a non-constructive way). It is derived by plugging into the
  Gilbert-Varshamov bound for binary codes $d=n/2$ and to make a Taylor
  approximation on the resulting formula. The construction uses Theorem~\ref{thedirect}
  directly. It guarantees for all $\epsilon$ and $k$ the existence of
  multisets with $$\mathcal{X} \subseteq  {\mathbb{F}}_2^{\Omega(k)}, {\mbox{\ \ \ }}
  | {\mathcal{X}}|=\bigO{\frac{k}{\epsilon^2}}.$$
\item The linear programming bound tells us that we cannot hope to
  produce $\epsilon$-bias spaces with $$\mathcal{X} \subseteq  {\mathbb{F}}_2^{\Omega(k)}, {\mbox{\ \ \ }}
  | {\mathcal{X}}|=
\bigO{
    \frac{k}{\epsilon^2\log(1/\epsilon)}} .$$
\end{itemize}
One way of comparing the above results is to choose
$\epsilon=k^{-\alpha}$, $\alpha \in {\mathbb{R}}^+$ and then to take the logarithm with base $k$. The
bigO notation suggests that we then let $k$ go to infinity. The origin
of this 
point of view is~\cite[Sec.\ 1]{hermitian}. When making the above operation
we must be careful to specify which choices of $\alpha$ are
allowed. We remind the reader of the little-o notation. Given
functions $f_i(x):{\mathbb{Z}}^+ \rightarrow {\mathbb{R}}^+, i=1,2$ by 
$f_1(x)=o(f_2(x))$ we mean that for every choice of $c \in {\mathbb{R}}^+$
there exists a $\kappa(c) \in {\mathbb{Z}}^+$ such that when
$\kappa(c) < x$ then necessarily $f_1(x)\leq c f_2(x)$.
We have:
\begin{itemize}
\item RS-bound: The family of concatenated codes from Theorem~\ref{theconc} with
  Reed-Solomon codes as outer codes gives
$$\log_k(| {\mathcal{X}}|)=2+2\alpha+ o(1)$$ for all choices of $\alpha
\in {\mathbb{R}}^+$.
\item AG-bound: The family of concatenated codes from Theorem~\ref{theconc} with
  algebraic geometric codes as outer codes and $g < m$ 
gives  $$\log_k(| {\mathcal{X}}|)=1+3\alpha+o(1)$$ for all choices of $\alpha
\in {\mathbb{R}}^+$.
\item BT-bound: The family of concatenated codes from Theorem~\ref{theconc} with
  Hermitian codes as outer codes and $m < g$ 
gives 
$$\log_k(| {\mathcal{X}}|)=\frac{5}{4}+\frac{5}{2} \alpha+o(1)$$ for all choices of $\alpha \in ]1/2,\infty[$.
\item The family of concatenated codes from norm-trace codes of low
  dimension gives  
$$\log_k(| {\mathcal{X}}|)=\frac{l+1}{l}(1+\alpha(l-\sqrt{l}))+o(1)$$
for $l=4, 5,\ldots$, and for  all $\alpha \in
[1/\sqrt{l},\infty[$ (see Section~\ref{normtrace}).
\item The Gilbert-Varshamov bound and the Linear Programming bound in
  combination tell us that we can achieve
  $$\log_k(| {\mathcal{X}}|)=1+2\alpha+o(1)$$
for all choices of $\alpha \in {\mathbb{R}}^+$  
but no better than this.
\end{itemize}
In the present paper we shall introduce a new family of small-bias spaces
using a combination of Hermitian codes as outer code. This family 
gives  
$$\log_k(| {\mathcal{X}}|)=\frac{4}{3}+\frac{8}{3}
\alpha  + o(1)$$ for all choices of $\alpha
\in {\mathbb{R}}^+ $.
We
allow $2g<m$ and it is therefore surprising that for $\alpha \in
]1,\infty[$ the achievements are better than those of the 
 Hermitian
codes with $g<m$. Our small-bias spaces perform  better
than the ones derived from norm-trace codes for all $l\geq 5$ (see
Section~\ref{normtrace} for the proof). For
$\alpha < 1$ they behave better than what can be achieved using
Reed-Solomon codes as outer code. For $\alpha < 1$ admittedly the new
$\epsilon$-bias spaces perform worse than the spaces coming from the
AG construction. This, however,  is only part of the picture. It turns
out that to construct the spaces with $\alpha < 1/2$ from the AG
construction requires quite a number of operations. In contrast, our construction
is considerable faster. We shall revert to this issue in Section~\ref{secfast}. 
Before dealing with the new
construction we will investigate how to ensure
$\epsilon=k^{-\alpha}$ in the case of the AG bound. It turns out that
for $\alpha< 1/2$ the situation is rather complicated. We include the
description here, as to our best knowledge, the details cannot be found in
the literature.
\section{The AG-bound}\label{sechermitian}
Let $q$ be a power of $2$ and
consider an algebraic function field over ${\mathbb{F}}_{q^2}$ of genus $g$ with
at least ${\mathcal{N}}=(q-1)g$ rational places. That is, the function field attains the
Drinfeld-Vladut bound. 
As noted in the introduction 
Theorem~\ref{theconc} equipped with a one-point algebraic geometric
code from the above function field produces $\epsilon$-bias
spaces ${\mathcal{X}}\subseteq {\mathbb{F}}_2^{\Omega (k)}$ with
\begin{equation}
| {\mathcal{X}}| = \bigO{\frac{k}{\epsilon^3\log_2(\frac{1}{\epsilon})}}.\label{eqeq}
\end{equation}
In the following we investigate how to achieve corresponding values
$\epsilon$ and $k$ under the requirement $\epsilon=k^{-\alpha}$,
$\alpha > 0$, and
$k\rightarrow \infty$. Observe, that in this situation for any fixed
$\alpha$ we have $\epsilon \rightarrow 0$. For
completeness we start by proving~(\ref{eqeq}) in this setting.\\
Consider rational places $P_1, \ldots , P_{{\mathcal{N}}-1},Q$ and let $U=P_1+\cdots
+P_{{\mathcal{N}}-1}$ and $G=(ag)Q$ with $a \geq 1$. The code $C_{\mathcal{L}}(U,G)$
has parameters $N=(q-1)g-1$, $K\geq \deg G-g=(a-1)g$, and $D\geq N-\deg
G=((q-1)-a)g-1$. As we are interested in asymptotics we shall assume
$N=(q-1)g$ and $D \geq ((q-1)-a)g$. From Theorem~\ref{theconc}
we get $\epsilon$-bias spaces with $\epsilon=a/(q-1)$,
${\mathcal{X}}\subseteq {\mathbb{F}}_2^{\Omega(k)}$. Here,
$k=2\log_2(q)(a-1)g$ and we have $|{\mathcal{X}}|=q^2N=(q^3-q^2)g$. As $a$ is
bounded below by $1$ and $\epsilon \rightarrow 0$ we need $q
\rightarrow \infty$ when $k\rightarrow \infty$. So the task basically
boils down to establishing a sequence of function fields over increasingly large fields and a corresponding function $a(q)$ such that
\begin{equation}
| {\mathcal{X}}|=\bigO{\frac{2 \log_2(q)(a-1)g}{\left(\frac{a}{q-1}\right)^3\log_2\left(\frac{q-1}{a}\right)}}.\label{equationen}
\end{equation}
Note that the argument on the right side is a function in the single
variable $q$ as by construction now $g$ is a function of $q$. We have 
$$\frac{2\log_2(q)(a-1)g}{\left(\frac{a}{q-1}\right)^3\log_2\left(
    \frac{q-1}{a}\right)} \geq \frac{1}{2}
\frac{\log_2(q)(a-1)}{a^3(\log_2(q-1)-\log_2(a))} | {\mathcal{X}}|$$
as $(q-1)^3\geq \frac{1}{4}(q^3-q^2)$ holds for $q\geq 2$. In
conclusion~(\ref{equationen}) holds if $a(q)=\bigO{1}$.\\
We first assume that the sequence of function fields are the
Hermitians which are function fields with $g=q(q-1)/2$. Here, actually the number of rational
places is $2qg+q^2+1$ but we shall only use $(q-1)g$ of them. Let
$a=1+q^{-c}$ where $0\leq c <2$. Clearly, $a(q)=\bigO{1}$ as
requested. We have 
$k=2\log_2(q)q^{-c}g=q^{2-c}q^\beta$
where $\beta(q) \rightarrow 0$ for $q\rightarrow \infty$. Hence,
asymptotically $\epsilon=k^{-\alpha}$ with $\alpha=1/(2-c)$. In other
words the situation is clear for $\alpha \in [\frac{1}{2},\infty[$.\\
To achieve $\alpha \in ]0,\frac{1}{2}[$ is more difficult. The problem
is to keep $a(q)=\bigO{1}$ at the same time as having
$\epsilon=k^{-\alpha}$. For this purpose we consider families of
towers of function fields over ${\mathbb{F}}_{q^2}$ attaining the
Drinfeld-Vladut bound \cite{garciastichtenoth}. We will need one tower
for each value of $q$. Note that in such a
tower for arbitrary $v \geq 2$ we can find a function field with $g \geq
q^v$. Say $g=q^{v+d(q)}$, where $d(q) \geq 0$ holds.
Let $a(q)=1+q^{-d(q)}$ then clearly $a(q)=\bigO{1}$ holds. We have 
$k=2\log_2(q)(a-1)g=q^{v+\beta}$
where $\beta(q)\rightarrow 0$ for $q \rightarrow \infty$. Also
$\epsilon=q^{-1+\gamma}$ where $\gamma(q)\rightarrow 0$ for $q
\rightarrow \infty$.  Hence,
$k^{-\alpha}=\epsilon$ asymptotically means $v\alpha=1 \Rightarrow
\alpha=1/v$. As we only assumed $v \geq 2$ we have established that all
$\alpha \in ]0,\frac{1}{2}[$ can be attained. \\
For our purpose the best candidate for a family of
good towers of function fields is the second construction by Garcia
and Stichtenoth~\cite{garciastichtenoth}. In~\cite{gsbasis} it was
shown how to construct $C_{\mathcal{L}}(U,G)$ codes from this tower
using
\begin{equation}
\bigO{(N\log_q(N))^3}\label{eqfollows}
\end{equation}
operations over ${\mathbb{F}}_{q^2}$. Although we might only need codes of small dimension the
method as stated requests us to find  bases for all one-point codes. 
As shall be demonstrated in Section~\ref{secfast} 
the small-bias spaces of the present paper can be constructed much
faster than what~(\ref{eqfollows}) guarantees for the AG construction. 
\section{The new small-bias spaces}\label{secnew}
In the present paper we propose a new choice of outer codes in the
construction of Theorem~\ref{theconc}. As already mentioned this
results in small-bias spaces with good properties. The new choice of outer codes
is derived by combining two Hermitian codes as described below. 
The easiest way to explain the combination is by using
the language
of affine variety codes~\cite{lax} and we therefore
start our investigations with a presentation of Hermitian codes as
such.
\begin{definition}
Given a monomial ordering $\prec$ and an ideal $I \subseteq
{\mathbb{F}}[X_1, \ldots , X_m]$ (here ${\mathbb{F}}$ is any field)
the footprint is
\begin{eqnarray}
\Delta_\prec (I)&:=&\{X_1^{\alpha_1}\cdots X_m^{\alpha_m} \mid
X_1^{\alpha_1}\cdots X_m^{\alpha_m} {\mbox{ \ is not a leading
    monomial}} \nonumber \\ 
&&{\mbox{ \hspace{6cm} of any polynomial in \
  }}I\}.\nonumber 
\end{eqnarray}
\end{definition}
We have the following two useful results \cite[Pro.\ 4 and Pro.\ 8,
Sec.\ 5.3]{clo}.
\begin{theorem}\label{thebasis}
The set $\{M+I \mid M \in \Delta_\prec(I)\}$ is a basis for
${\mathbb{F}}[X_1, \ldots , X_m]/I$ as a vector space over ${\mathbb{F}}$.
\end{theorem}
As a corollary one gets the following result often referred to as the
footprint bound \cite{footprint,blahut}.
\begin{theorem}\label{thefootprint}
Assume $I$ is zero-dimensional (meaning that $\Delta_\prec(I)$ is
finite). The variety ${\mathbb{V}}_{\bar{{\mathbb{F}}}}(I)$
satisfies $| {\mathbb{V}}_{\bar{{\mathbb{F}}}}(I)| \leq |
\Delta_\prec(I)|.$
\end{theorem}
Consider the Hermitian polynomial $X^{q+1}-Y^q-Y$ and the
corresponding ideal 
$$I=\langle X^{q+1}-Y^q-Y\rangle \subseteq {\mathbb{F}}_{q^2}[X,Y].$$
Define a monomial function $w$ by $w(X)=q$ and $w(Y)=(q+1)$ and
consider the weighted degree monomial ordering $\prec_w$ given
by $X^{\alpha_1}Y^{\beta_1} \prec_w X^{\alpha_2}Y^{\beta_2}$ if one of
the following two conditions holds:
\begin{enumerate}
\item $w(X^{\alpha_1} Y^{\beta_1} )<w(X^{\alpha_2}Y^{\beta_2}) $.
\item $w(X^{\alpha_1} Y^{\beta_1} )=w(X^{\alpha_2}Y^{\beta_2}) $ 
but
  $\beta_1 < \beta_2$.
\end{enumerate}
Observe for later use that no two different monomials in
$$\Delta_{\prec_w}(I)=\{X^iY^j \mid 0 \leq i {\mbox{ \ and \ }} 0\leq
j<q\}$$
are of the same weight implying that $w: \Delta_{\prec_w}(I)
\rightarrow \langle q, q+1\rangle$ is a bijection. Observe also that
the Hermitian polynomial $X^{q+1}-Y^q-Y$ contains exactly two
monomials of highest weight. The implication of this is that 
$$w({\mbox{lm}}(F(X,Y))=w({\mbox{lm}}(F(X,Y) {\mbox{ rem }}\{X^{q+1}-Y^q-Y\})$$
holds for any polynomial  $F(X,Y)$ that possesses exactly one monomial of
highest weight 
in its support.\\
Consider next the ideal
$$I_{q^2}:=\langle X^{q^2}-X, Y^{q^2}-Y\rangle +I.$$
The variety
${\mathbb{V}}_{\mathbb{F}_{q^2}}(I)={\mathbb{V}}_{\mathbb{F}_{q^2}}(I_{q^2})$
consists of $n=q^3$ different points $\{P_1, \ldots  P_n\}$. The set
$\{X^{q^2}-X,X^{q+1}-Y^q-Y\}$ constitutes a Gr\"{o}bner basis for
$I_{q^2}$ with respect to $\prec_w$ and therefore
$$\Delta_{\prec_w}(I_{q^2})=\{ X^iY^j \mid 0\leq i<q^2, 0\leq j<q\}$$
holds. It now follows from Theorem~\ref{thebasis} that
$$\{ X^iY^j+I_{q^2} \mid 0\leq i<q^2, 0\leq j<q\}$$
is a basis for ${\mathbb{F}}_{q^2}[X,Y]/I_{q^2}$ as a vector space
over ${\mathbb{F}}_{q^2}$. The code construction relies on the
bijective evaluation map ${\mbox{ev}}: {\mathbb{F}}_{q^2}[X,Y]/I_{q^2}
\rightarrow {\mathbb{F}}_{q^2}^{n}$ given by
${\mbox{ev}}(F(X,Y)+I_{q^2})=(F(P_1), \ldots ,
F(P_n))$. Theorem~\ref{thefootprint} tells us that we can estimate the
Hamming weight of a word $\vec{c}={\mbox{ev}}(F(X,Y)+I_{q^2})$ by
$$w_H(\vec{c})\geq n-| \Delta_{\prec_w}(\langle F(X,Y)\rangle
+I_{q^2})|.$$
Without loss of generality we can assume ${\mbox{Supp}}(F) \subseteq
\Delta_{\prec_w}(I_{q^2})$. From the discussion prior to the
definition of $I_{q^2}$ we conclude that
no two different monomials in $F(X,Y)$ are of the same weight. As a
consequence
$$w({\mbox{lm}}(X^\alpha Y^\beta F(X,Y))=w({\mbox{lm}}(X^\alpha
Y^\beta F(X,Y) {\mbox{ rem }}\{X^{q+1}-Y^q-Y\})$$
holds for all $X^\alpha Y^\beta$. Write
$\Lambda=w(\Delta_{\prec_w}(I))=\langle q,q+1\rangle$, 
$\Lambda^\ast=w(\Delta_{\prec_w}(I_{q^2}))\subseteq \Lambda$ and
$\lambda=w({\mbox{lm}}(F)) \in \Lambda^\ast$.
We have
\begin{eqnarray}
| \Delta_{\prec_w}(\langle F(X,Y)\rangle +I_{q^2})|\leq |( \Lambda^\ast
-(\lambda +\Lambda))| \leq |  (\Lambda \backslash (\lambda +\Lambda)|=\lambda, \nonumber
\end{eqnarray}
where the last equality comes from~\cite[Lem.\ 5.15]{handbook}. Hence, $w_H(\vec{c}) \geq n-\lambda$ holds. Observe
that 
\begin{equation}
\Lambda^\ast = \{ \lambda_1, \ldots , \lambda_g\} \cup \{2g, \ldots , n-1\} \cup
\{ \lambda_{n-g+1},\ldots ,\lambda_n\}, \label{eqcharac}
\end{equation}
where $\lambda_i \leq g-1+i$ for $i=1, \ldots , g$. This is a general
result for Weierstrass semigroups and not particular for the
Hermitian function field. Having described the Hermitian codes as affine
variety codes we are now ready to introduce the combination of codes
on which our construction of small-bias spaces rely. Consider the ideal
$$I^{(2)}_{q^2}:=\langle X_1^{q+1}-Y_1^q-Y_1, X_2^{q+1}-Y_2^q-Y_2,X_1^{q^2}-X_1,Y_1^{q^2}-Y_1,X_2^{q^2}-X_2,Y_2^{q^2}-Y_2 \rangle$$
and the corresponding variety 
$${\mathbb{V}}_{\mathbb{F}_{q^2}}(I^{(2)}_{q^2})={\mathbb{V}}_{\mathbb{F}_{q^2}}(I_{q^2})
  \times {\mathbb{V}}_{\mathbb{F}_{q^2}}(I_{q^2})=\{Q_1, \ldots , Q_{q^6}\}.$$
Define a monomial function  $w^{(2)}$ given by
$w^{(2)}(X_1)=(q,0)$,
$w^{(2)}(Y_1)=(q+1,0)$,$w^{(2)}(X_2)=(0,q)$, and finally
$w^{(2)}(Y_2)=(0,q+1)$. Let $\prec_{\mathbb{N}_0^2}$ be any monomial
ordering on ${\mathbb{N}}_0^2$ and define $\prec_{w^{(2)}}$ by
$$X_1^{\alpha_1^{(1)}}Y_1^{\beta_1^{(1)}}X_2^{\alpha_1^{(2)}}Y_2^{\beta_1^{(2)}}
\prec_w^{(2)}
X_1^{\alpha_2^{(1)}}Y_1^{\beta_2^{(1)}}X_2^{\alpha_2^{(2)}}Y_2^{\beta_2^{(2)}}$$
if one of the following two conditions holds:
\begin{enumerate}
\item $w^{(2)}(X_1^{\alpha_1^{(1)}}Y_1^{\beta_1^{(1)}}X_2^{\alpha_1^{(2)}}Y_2^{\beta_1^{(2)}})
\prec_{\mathbb{N}_0^2}
w^{(2)}(X_1^{\alpha_2^{(1)}}Y_1^{\beta_2^{(1)}}X_2^{\alpha_2^{(2)}}Y_2^{\beta_2^{(2)}})$
\item $w^{(2)}(X_1^{\alpha_1^{(1)}}Y_1^{\beta_1^{(1)}}X_2^{\alpha_1^{(2)}}Y_2^{\beta_1^{(2)}})
=
w^{(2)}(X_1^{\alpha_2^{(1)}}Y_1^{\beta_2^{(1)}}X_2^{\alpha_2^{(2)}}Y_2^{\beta_2^{(2)}})$\\
but \\
$X_1^{\alpha_1^{(1)}}Y_1^{\beta_1^{(1)}}X_2^{\alpha_1^{(2)}}Y_2^{\beta_1^{(2)}}
\prec_{{\mbox{lex}}}
X_1^{\alpha_2^{(1)}}Y_1^{\beta_2^{(1)}}X_2^{\alpha_2^{(2)}}Y_2^{\beta_2^{(2)}}$.
\end{enumerate}
Here, $X_1\succ_{{\mbox{lex}}}Y_1 \succ_{{\mbox{lex}}} X_2\succ_{{\mbox{lex}}}Y_2$
is assumed.
The set $\{ X_1^{q+1}-Y_1^q-Y_1, X_2^{q+1}-Y_2^q-Y_2,X_1^{q^2}-X_1,
X_2^{q^2}-X_2\}$ is a Gr\"{o}bner basis for $I_{q^2}^{(2)}$ with
respect to $\prec_{w^{(2)}}$ giving us the basis 
$$\{
X_1^{i_1}Y_1^{j_1}X_2^{i_2}Y_2^{j_2}+I_{q^2} \mid 0\leq i_1, i_2 < q^2, 0 \leq
j_1,j_2 < q\}$$
for ${\mathbb{F}}_{q^2}[X_1,Y_1,X_2,Y_2]/I_{q^2}^{(2)}$ as a
vectorspace over ${\mathbb{F}}_{q^2}$. For the code construction we
need the following bijective evaluation map
$${\mbox{EV}}: {\mathbb{F}}_{q^2}[X_1,Y_1,X_2,Y_2]/I^{(2)} \rightarrow
{\mathbb{F}}_{q^2}^{q^6}$$
given by ${\mbox{EV}}(F(X_1,Y_1,X_2,Y_2)+I_{q^2}^{(2)})=(F(Q_1,) ,\ldots ,
F(Q_{q^6}))$.
Define $\Lambda^{(2)}=\Lambda \times \Lambda$ and $\big(\Lambda^{(2)}\big)^\ast=\Lambda^\ast \times \Lambda^\ast$. We have 
$$\big(
\Lambda^{(2)}\big)^\ast=
w^{(2)}(\Delta_{\prec_{w^{(2)}}}(I^{(2)}_{q^2}))$$
where no two monomials in $\Delta_{\prec_{w^{(2)}}}(I^{(2)}_{q^2})$ have
the same weight. Similar to the situation of a Hermitian code we
consider a codeword
$\vec{c}={\mbox{EV}}(F(X_1,Y_1,X_2,Y_2)+I_{q^2}^{(2)})$ where without
loss of generality we will assume that $F(X_1,Y_1,X_2,Y_2)\in
\Delta_{\prec_{w^{(2)}}}(I_{q^2}^{(2)})$. We
write $\lambda^{(2)}=(\lambda_1,\lambda_2)=w^{(2)}({\mbox{lm}}(F)).$
We can estimate
\begin{eqnarray}
| \Delta_{\prec_{w^{(2)}}}(\langle F(X_1,Y_1,X_2,Y_2)\rangle
+I_{q^2}^{(2)})| &\leq&| \Lambda^{(2)} -(\lambda^{(2)}+\Lambda^{(2)})|
\nonumber \\
&\leq &q^6-(q^3-\lambda_1)(q^3-\lambda_2).\nonumber
\end{eqnarray}
Hence, $w_H(\vec{c})\geq (q^3-\lambda_1)(q^3-\lambda_2)$.\\
Consider the code $\tilde{E}(\delta)$ which is to Hermitian codes what
Massey-Costello-Justesen codes \cite{masseycostellojustesen} are to Reed-Solomon codes
\begin{eqnarray}
\tilde{E}(\delta):=&&{\mbox{Span}}_{{\mathbb{F}}_{q^2}} \bigg\{
    {\mbox{EV}}(X_1^{i_1}Y_1^{j_1}X_2^{i_2}Y_2^{j_2}+I_{q^2}^{(2)}) \mid
    0  \leq i_1, i_2 < q^2, 0 \leq j_1, j_2 < q ,\nonumber \\
&&{\mbox{ \hspace{2cm}}} (q^3-w(X_1^{i_1}Y_1^{j_1}))(q^3-w(X_2^{i_2}Y_2^{j_2})) \geq \delta\bigg\}.\nonumber
\end{eqnarray}
From our discussion we conclude that the minimum distance satisfies
$d(\tilde{E}(\delta))\geq \delta$. To estimate the dimension we make
use of the characterization~(\ref{eqcharac}). The task is to estimate
the number of $(\lambda_1,\lambda_2)$s that satisfies
$(q^3-\lambda_1)(q^3-\lambda_2)\geq \delta$. For this purpose we can
replace $\Lambda^\ast$ with
$$\{g, g+1, \ldots, q^3-1\}\cup \{\lambda_{n-g+1}, \ldots , \lambda_n\}.$$
When estimating the dimension $k(\tilde{E}(\delta))$ we shall
furthermore ignore the elements in $\{\lambda_{n-g+1}, \ldots
,\lambda_n\}$. Writing $T=q^3-g$ we thereby get
\begin{eqnarray}
k(\tilde{E}(\delta))&\geq&| \{(i,j) \mid 0\leq i,j \leq T-1,
(T-i)(T-j) \geq \delta\}|\nonumber \\
&\geq&\int_{0}^{T-\frac{\delta}{T}}
\int_0^{T-\frac{\delta}{T-i}}djdi=T^2-\delta+\d {\mbox{ln}} \left(
  \frac{\delta}{T^2}\right), \nonumber 
\end{eqnarray}
where the last inequality holds under the assumption $\delta \geq T$.
\begin{proposition}\label{propar}
Assume $\delta \geq T$ where $T=q^3-g$. The parameters of
$\tilde{E}(\delta)$ are $[n=q^6,k\geq T^2-\delta+\delta \ln
(\delta/T^2),d \geq \delta]$.
\end{proposition}
In~\cite{hyperbolictype} Feng-Rao improved codes $\tilde{C}(\delta)$
over ${\mathbb{F}}_{q^2}[X_1, Y_1,X_2,Y_2]/I_{q^2}^{(2)}$ were
considered and a formula similar to the above proposition was derived
under a stronger assumption on $\delta$. Feng-Rao improved codes are
described by means of their parity check matrix which is not very
useful when the aim is to construct a small-bias space. This is why we
included the description of $\tilde{E}(\delta)$ in the present
paper. We have a proof that $\tilde{E}(\delta)=\tilde{C}(\delta)$,
however, 
we do not include it here as it has no implication for the
construction of small-bias spaces. Observe that to derive
Proposition~\ref{propar} we did  not use detailed information about
the Weierstrass semigroup $\Lambda$ but relied only on the
genus and the number of roots of the Hermitian polynomial. Proposition~\ref{propar} can be
generalized to hold for not only two copies of Hermitian function fields but to
arbitrary many such copies. Such constructions, however, are not
useful when dealing with small-bias spaces so we do not treat them
here.\\
From Proposition~\ref{propar} and Theorem~\ref{theconc} we get a new
class of $\epsilon$-bias spaces:
\begin{theorem}\label{theC}
For any $\epsilon$, $0<\epsilon <1$ using codes $\tilde{E}(\delta)$ as
outer code in the construction of Theorem~\ref{theconc} one can
construct $\epsilon$-bias spaces with
\begin{equation}
{\mathcal{X}}\subseteq {\mathbb{F}}_2^{\Omega(k)},
{\mbox{ \ \ \ }}| {\mathcal{X}}|=\bigO{\left(\frac{k}{\epsilon+(1-\epsilon)\ln
      (1-\epsilon)}\right)^{\frac{4}{3}}}.\label{eqhurra}
\end{equation}
\end{theorem}
\begin{proof}
In the following we will use the substitution $1-\epsilon=\delta/N$
which follows from $\epsilon=(N-\delta)/N$. Assume $\delta >
\sqrt{N}$. We then have $\delta > T$ which is the condition in
Proposition~\ref{propar}. Note that $\delta > \sqrt{N}$ is equivalent
to $\epsilon < 1-(1/\sqrt{N})$. For $N \rightarrow \infty$ this
becomes $\epsilon < 1$ which is actually no restriction at all. From
the proposition we get
\begin{eqnarray}
\frac{K}{N}&\geq&\left(\frac{q^3-g}{q^6}\right)^2-\frac{\delta}{q^6}+\frac{\delta}{q^6}\ln
\left( \frac{\delta}{(q^3-g)^2}\right) \nonumber \\
&\geq&o(1)+1-(1-\epsilon)+(1-\epsilon)\ln (1-\epsilon)\nonumber \\
&=&o(1)+\epsilon+(1-\epsilon) \ln (1-\epsilon).\nonumber
\end{eqnarray}
With $q^2=2^s$ we have
$$|{\mathcal{X}}| \leq \frac{2^s}{s}\left( \frac{k}{o(1)+\epsilon+(1-\epsilon)\ln
    (1-\epsilon)}\right).$$
But $|{\mathcal{X}}|=(2^s)^4$ implies $2^s=| {\mathcal{X}}|^{1/4}$ and (\ref{eqhurra}) has been demonstrated.
\end{proof}

\begin{theorem}
Consider the family of $\epsilon$-bias spaces in
Theorem~\ref{theC}. Given $\alpha \in {\mathbb{R}}^+$ choose
$\epsilon=k^{-\alpha}$ and let $k \rightarrow \infty$. We have
\begin{equation}
\log_{k}(|{\mathcal{X}}|)=
\frac{4}{3}+\frac{8}{3} \alpha+o(1).\label{eqhiphip}
\end{equation}
\end{theorem}
\begin{proof}
We have 
$$\log_k(|{\mathcal{X}}|) \leq
\frac{4}{3}-\frac{4}{3}\log_k(\epsilon+(1-\epsilon)\ln
(1-\epsilon)).$$
We now apply Taylors formula to derive 
$\ln (1-\epsilon) =-\epsilon-{\epsilon^2}/{2(1-c)^2}$ 
for some $c \in [0,\epsilon]$. This produces
\begin{eqnarray}
\log_k(|{\mathcal{X}}|)&\leq &\frac{4}{3}
-\frac{4}{3}\log_k\left(\epsilon+(1-\epsilon)(-\epsilon-\frac{\epsilon^2}{2(1-\epsilon)^2})\right)\nonumber
\\
&\leq&\frac{4}{3}-\frac{4}{3} \log_k \left(\epsilon^2\left(\frac{2(1-\epsilon)^2-\epsilon^2}{(1-\epsilon)^2}\right)\right).\nonumber
\end{eqnarray}
With $\epsilon=k^{-\alpha}$ we arrive at~(\ref{eqhiphip}).
\end{proof}

\section{Time complexity considerations}\label{secfast}
To build the multiset ${\mathcal{X}}$ in our construction we need to
construct a generator matrix for the concatenated code. This involves the
following tasks:
\begin{enumerate}
\item Build the generator matrix $G_1$ for $\tilde{E}(\delta)$.
\item Express every entry of $G_1$ as a binary vector giving us $G_2$
  (a matrix with binary vectors as entries).
\item For every row in $G_2$ we produce $s=\log_2(q^2)$ rows. This
  is done by taking cyclic shifts of all the vectors appearing in the
  row. We arrive at a matrix $G_3$.
\item Every entry in $G_3$ is a vector of length $s$ and it must be
  multiplied with the $s \times 2^s$ generator matrix of the
  Walsh-Hadamard code producing $G_4$.
\end{enumerate}
The total cost in binary operations is estimated as follows:
\begin{enumerate}
\item Determining functions and points for the code construction is
  inexpensive. To produce one entry costs
  $\bigO{\log(N)\log(\log(N))}$ operations. $G_1$ is a $K\times N$
  matrix. Using $K\leq N-D+1$, $\epsilon=(N-D)/N$,
  $\epsilon=k^{-\alpha}$, and $k=K\log_2(N)/6$ we arrive at
$K \leq N^{\frac{1}{1+\alpha}}(\log_2(N))^{\frac{-\alpha}{1+\alpha}}6^{\frac{\alpha}{1+\alpha}}.$
So the price for building $G_1$ is
$\bigO{N^{\frac{2+\alpha}{1+\alpha}}(\log(N))^{\frac{1}{1+\alpha}}\log(\log(N))}$.
\item To produce one entry in $G_2$ costs
  $\bigO{N^{\frac{1}{3}}\log(N^{\frac{1}{3}})\log(\log(N^{\frac{1}{3}}))}$
  operations. That is, to produce $G_2$ from $G_1$ amounts to\\
  $\bigO{N^{\frac{7+4\alpha}{3+3\alpha}}\log(N)^{\frac{1}{1+\alpha}}\log(\log(N))}$
  operations.
\item There will be
  $\bigO{N^{\frac{2+\alpha}{1+\alpha}}(\log(N))^{\frac{1}{1+\alpha}}}$
  entries in $G_3$ each coming with a cost of $s$
  operations. Altogether we have
  $\bigO{N^{\frac{2+\alpha}{1+\alpha}}(\log(N))^{\frac{2+\alpha}{1+\alpha}}}$
  operations.
\item The price for multiplying with a generator matrix for the
  Walsh-Hadamard code is $N^{\frac{1}{3}}\log(N)$ giving a total cost
  of
\begin{equation}
  \bigO{N^{\frac{7+4\alpha}{3+3\alpha}}(\log(N))^{\frac{2+\alpha}{1+\alpha}}}\label{eqpris}
\end{equation}
  operations for producing $G_4$ from $G_3$.
\end{enumerate}
Clearly, the overall cost is that
of~(\ref{eqpris}). Note that~(\ref{eqpris}) counts binary operations
in contrast to~(\ref{eqfollows}) which counts operations in ${\mathbb{F}}_{q^2}$.

\section{Small-bias spaces from norm-trace codes}\label{normtrace}
The method developed by Ben-Aroya and Ta-Shma for Hermitian codes
in~\cite{hermitian} were generalized to norm-trace codes by Matthews
and Peachey in~\cite{MP}. Given $r \geq 2$ consider the $C_{ab}$ curve
\cite{cab} 
$$X^{\frac{q^r-1}{q-1}}-Y^{q^{r-1}}-Y^{q^{r-2}}-\cdots -Y^q-Y$$
known as the norm-trace curve over ${\mathbb{F}}_{q^r}$
\cite{normtrace}. Clearly, $r=2$ corresponds to the Hermitian
function field. The following theorem from~\cite{MP} coincides with~(\ref{eqBT}) when $l=4$.
\begin{theorem}\label{theA}
Given an integer $l$, $l\geq 4$, define $r=\lfloor
(l+2)/3\rfloor$. Let $k$ be a positive integer and $\epsilon$ a real
number, $0<\epsilon <1$ such that 
\begin{equation}
\frac{\epsilon}{\big( \log_v(1/\epsilon)\big)^{\frac{1}{\sqrt{l}}}} \leq
  k^{\frac{-1}{\sqrt{l}}} \label{eqtrekant}
\end{equation}
holds. Here, $v$ is any fixed real number larger than $1$. Using the
norm-trace function field over ${\mathbb{F}}_{q^r}$ one can construct an
$\epsilon$-bias space  ${\mathcal{X}}\subseteq
{\mathbb{F}}_2^{\Omega(k)}$ with
$$
|{\mathcal{X}}| =\bigO{  \left( \frac{k}{\epsilon^{l-\sqrt{l}}\log_v(1/\epsilon)}\right)^{\frac{l+1}{l}}}.$$ 
\end{theorem}
In the above theorem it is not completely clear how well the cases
$l \geq 5$ compete with the case $l=4$. Below we address this question
and also compare the small-bias spaces from Theorem~\ref{theA} with
those achieved by using the codes $\tilde{E}(\delta)$ as is done
in the present paper.\\
We first translate Theorem~\ref{theA} into the setting from
Section~\ref{secintro} where for increasing $k$ and fixed $\alpha$ we
consider a sequence of $\epsilon$-bias multisets with
$\epsilon=k^{-\alpha}$. Condition~(\ref{eqtrekant}) from
Theorem~\ref{theA} then translates into
$$k^{1-\alpha \sqrt{l}}\leq \alpha \log_v(k).$$
For fixed $v$, $\log_v(k)=\bigO{k^\beta}$ holds for any $\beta
>0$. Therefore we have
$$1-\alpha\sqrt{l}\leq \log_k(\alpha).$$
Letting $k \rightarrow \infty$ we get the condition 
$$\frac{1}{\sqrt{l}} \leq \alpha.$$
Theorem~\ref{theA} therefore guarantees that for any $\alpha \geq
1/\sqrt{l}$ we can construct an infinite sequence of $\epsilon$-bias
spaces with $\epsilon=k^{-\alpha}$, ${\mathcal{X}}\subseteq
{\mathbb{F}}_2^{\Omega(k)}$ such that
\begin{equation}
\log_k (| {\mathcal{X}}|)
= \frac{l+1}{l}(1+\alpha(l-\sqrt{l}))+o(1).\label{eqdiamond}
\end{equation}
Given an $\alpha$ and two integers $l_1, l_2 \geq 4$ with $\alpha \geq
1/\sqrt{l_i}$, $i=1,2$ it is clear from~(\ref{eqdiamond}) that the
best result is obtained by choosing the smallest $l_i$. So the
advantage of Theorem~\ref{theA} over~(\ref{eqBT}) boils down to
the fact that Theorem~\ref{theA} allows for any $\alpha$
provided that the $l$ is chosen accordingly while~(\ref{eqBT}) requires $\alpha \geq 1/2$. Recall from Section~\ref{secnew}
that using the code $\tilde{E}(\delta)$ in the construction of
Theorem~\ref{theconc} one achieves
\begin{equation}
\log_k (| {\mathcal{X}}| )=
\frac{4}{3}+\frac{8}{3} \alpha +o(1) \label{eqcircle}
\end{equation}
for any choice of $\alpha$. We now compare this result
with~(\ref{eqdiamond}) ignoring of course the $o(1)$ parts. For fixed $l$ (\ref{eqdiamond}) is a linear
expression in $\alpha$ which is smaller than the linear expression
from~(\ref{eqcircle}) when $\alpha=0$. We now show that for
$\alpha=1/\sqrt{l}$ (which is the smallest $\alpha$ allowed)
(\ref{eqdiamond}) is larger than (\ref{eqcircle}) when $l\geq 5$. It
follows that none of  the cases $l \geq 5$ can compete with the
construction of the present paper. To show that (\ref{eqdiamond}) is
larger than (\ref{eqcircle}) for $\alpha=1/\sqrt{l}$ we substitute $k=\sqrt{l}$ into
(\ref{eqdiamond})-(\ref{eqcircle}) to get
$$\frac{1}{k^2}(k^3-\frac{4}{3}k^2-\frac{5}{3}k).$$
The function $k^3-\frac{4}{3}k^2-\frac{5}{3}k$ is positive for $k$
belonging to the interval from $0$ to approximately $2.119$ and
negative for higher values of $k$. Therefore for all $l\geq 5$ indeed
(\ref{eqcircle}) is better than (\ref{eqdiamond}).

\section{Acknowledgments}
The present work was done while Ryutaroh Matsumoto was visiting Aalborg
University as a Velux Visiting Professor supported by the Villum
Foundation. The authors gratefully acknowledge this support. The authors also gratefully acknowledge the support from
the Danish National Research Foundation and the National Science
Foundation of China (Grant No.\ 11061130539) for the Danish-Chinese Center for Applications of Algebraic Geometry in Coding Theory and Cryptography.

\end{document}